\documentclass[11pt]{amsart}

\usepackage{xspace}

\usepackage{amsmath}
\usepackage{amsthm}
\usepackage{amssymb}

 \usepackage{url}
 \usepackage{graphicx} 
 \usepackage[all]{xy} 

\allowdisplaybreaks 

\newtheorem{theorem}{Theorem}[section]

\theoremstyle{definition}
\newtheorem{definition}[theorem]{Definition}
\newtheorem{example}[theorem]{Example}

\theoremstyle{remark}
\newtheorem{remark}{Remark}

\numberwithin{figure}{section}

\newcommand{\RR}    {\Bbb{R}}

\newcommand{\id}    {\mathrm{id}}

\newcommand{\im}    {\mathrm{Im}}

\newcommand{\cat}   {\mathrm{cat}}

\newcommand{\scat}    {\mathrm{scat}}

\newcommand{\st}    {\mathrm{st}}
\newcommand{\lk}    {\mathrm{lk}}
\newcommand{\scrit}    {\mathrm{scrit}}
\newcommand{\Crit}    {\mathrm{crit}}

\title[A discrete version of the L--S Theorem]{Strong discrete Morse theory and simplicial L--S category: A discrete version of the Lusternik-Schnirelmann Theorem.}          

\author[D. Fdez-Ternero \and E. Mac\'{\i}as-Virg\'os \and N. A. Scoville \and J. A. Vilches]{D. Fern\'{a}ndez-Ternero \and E. Mac\'{\i}as-Virg\'os \and N. A. Scoville \and J. A. Vilches}   

\thanks{The first and the fourth authors were partially supported by MINECO and FEDER Research Project MTM2015-65397-P and Junta de Andaluc\'{\i}a Research Groups FQM-326 and FQM-189. The second  author was partially supported by MINECO and FEDER Research Project MTM2016-78647-P}

\date{\today}

\begin{document}

\begin{abstract}
We prove a discrete version of the Lusternik--Schnirelmann theorem
for discrete Morse functions and the recently introduced
simplicial Lusternik--Schnirelmann category of a simplicial
complex. To accomplish this, a new notion of critical object of a discrete
Morse function is presented, which generalizes the usual concept
of critical simplex (in the sense of R. Forman). We show that
the non-existence of such critical objects guarantees the strong
homotopy equivalence (in the Barmak and Minian's sense) between
the corresponding sublevel complexes. Finally, we establish that
the number of critical objects of a discrete Morse function
defined on $K$ is an upper bound for the non-normalized simplicial
Lusternik--Schnirelmann category of $K$.
\end{abstract}


\maketitle

\setcounter{tocdepth}{1}
\tableofcontents

\section{Introduction}

Since its inception by Robin Forman \cite{F-98}, discrete Morse
theory has been a powerful and versatile tool used not only in
diverse fields of mathematics, but also in applications to other
areas \cite{Real-15} as well as a computational tool \cite{C-G-N2016}. Its
adaptability stems in part from the fact that it is a discrete
version of the beautiful and successful ``smooth'' Morse theory
\cite{M-63}. While smooth Morse theory touches many branches of
math, one such area that has its origins in critical point theory
is that of Lusternik--Schnirelmann category.  The (smooth)
Lusternik--Schnirelmann category or L--S category of a smooth
manifold $X$, denoted $\cat(X)$, was first introduced in
\cite{L-S}, where the authors proved what is now known as the
Lusternik--Schnirelmann Theorem (see also \cite{CLOT} for a
detailed survey of the topic). A version of this result can be
stated as follows:

\begin{theorem} If $M$ is a compact smooth manifold and $f\colon M \to \RR$ is a smooth map, then
$$\cat(M)+1\leq \sharp(\Crit(f))$$

where $\Crit{(f)}$ is the set of all critical points of the
function $f$.

\end{theorem}

There are many ``smooth'' versions of this theorem in various
contexts (see for example \cite{Palais-66}). The aim of this paper
is to view Forman's discrete Morse theory from a different
perspective in order to prove a discrete version of the L--S theorem
compatible with the recently defined simplicial L--S category
developed by three of the authors \cite{F-M-V}. This simplicial
version of L--S category is suitable for simplicial complexes. Other
attempts have been made to develop such a ``discrete'' L--S
category.  In \cite{AS}, one of the authors developed a discrete
version of L--S category and proved an analogous L--S theorem for
discrete Morse functions. Our version of the L--S Theorem, Theorem
\ref{sLS theorem}, relates a new generalized notion of critical
object of a discrete Morse function to the simplicial L--S category
of \cite{F-M-V}.

In this paper we use the notion of simplicial L--S category defined
in {\cite{F-M-V}}. This simplicial approach of L--S category uses
strong collapses in the sense of Barmak and Minian \cite{B-M-12}
as a framework for developing categorical sets.  As opposed to
standard collapses, strong collapses are natural to consider
in the simplicial setting since they correspond to simplicial maps
(see Figure~\ref{strong_collapse}).

\begin{figure}[htbp]
\begin{center}
\includegraphics[width=.5\linewidth]{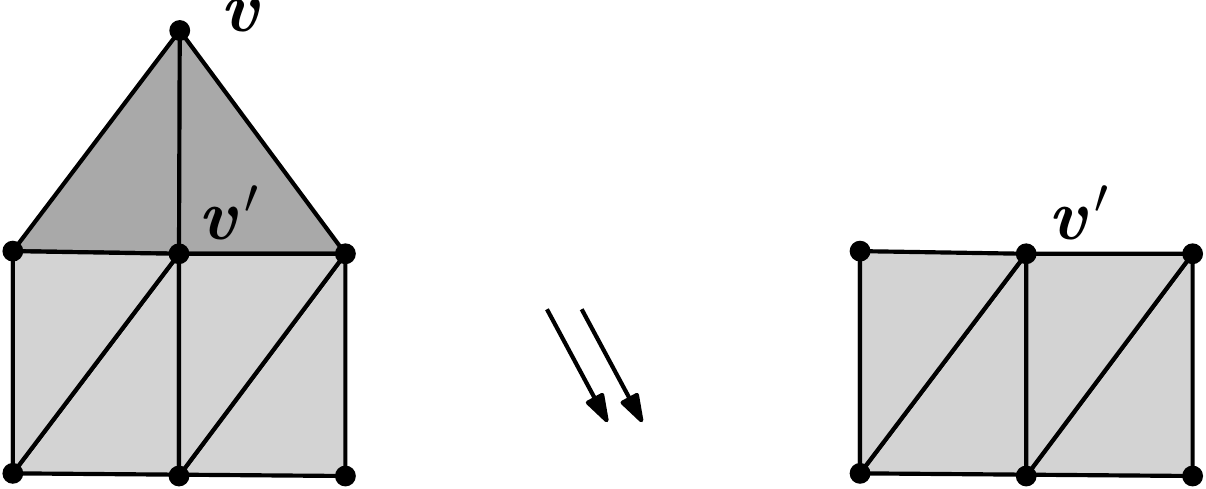}
\caption{An elementary strong collapse from $K$ to $K-\{v\}$.}
\label{strong_collapse}
\end{center}
\end{figure}

This is especially contrasted with (standard) elementary collapses, which in
general do not correspond to simplicial maps (see
Figure~\ref{standard_collapse}).

\begin{figure}[htbp]
\begin{center}
\includegraphics[width=.5\linewidth]{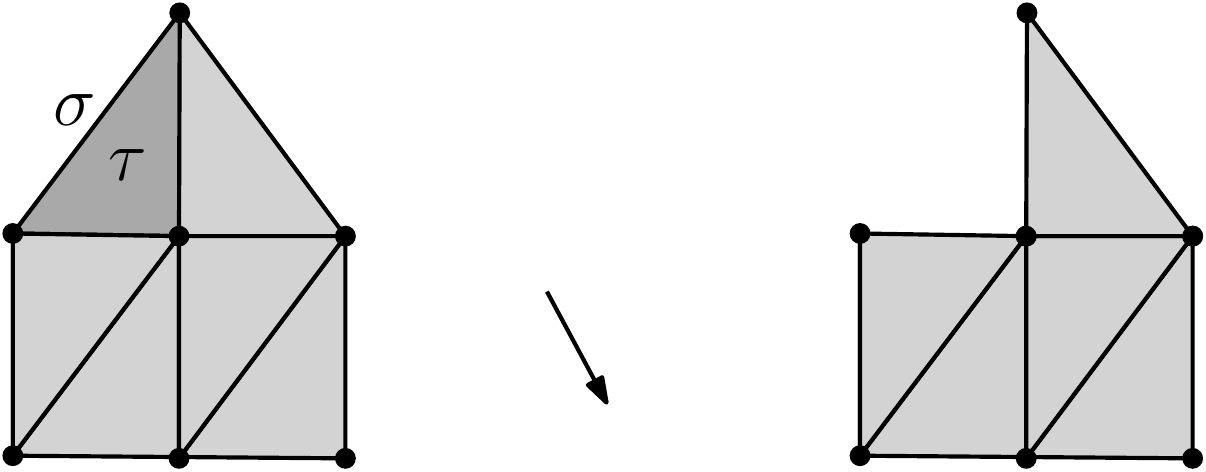}
\caption{An elementary (standard) collapse from $K$ to $K-\{\sigma , \tau\}$.}
\label{standard_collapse}
\end{center}
\end{figure}

Furthermore, elementary strong collapses correspond to the
deletion of the open star of a dominated vertex. Notice that in
general, the deletion of a vertex is not a simplicial map.

\begin{figure}[htbp]
\begin{center}
\includegraphics[width=.6\linewidth]{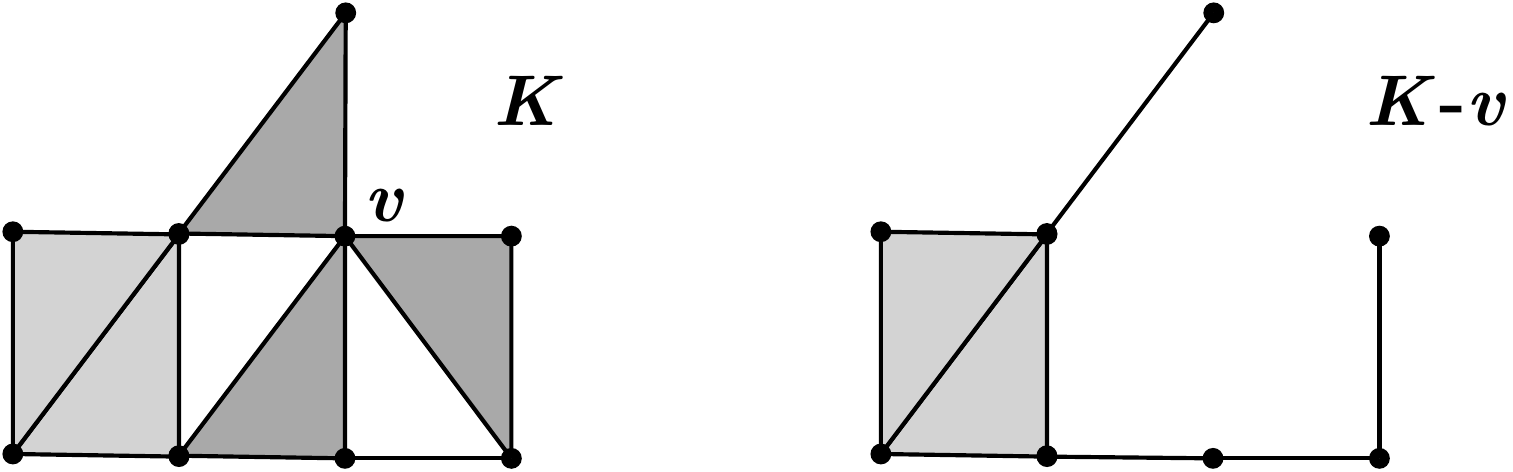}
\caption{Deletion of a vertex that does not correspond to a simplicial map.}
\label{vertex_deletion}
\end{center}
\end{figure}

This paper is organized as follows.  Section \ref{Fundamentals of
simplicial complexes} contains the necessary background and basics
of simplicial complexes and collapsibility. Section \ref{Strong
discrete Morse theory} is devoted to both reviewing discrete Morse
theory and introducing a generalized notion of critical object in
this context. Here we develop a collapsing theorem for discrete
Morse functions which is analogous to the classical result of
Forman (Theorem 3.3 of \cite{F-98}).

Section \ref{simplicial Lusternik--Schnirelmann category} is the
heart of the paper. In this section, we recall the definitions and
basic properties of the simplicial L--S category and prove the
simplicial L--S theorem in Theorem \ref{sLS theorem}. The rest of
the section is devoted to examples and immediate applications.

\section{Fundamentals of simplicial complexes}\label{Fundamentals of simplicial complexes}

In this section we review some of the basics of simplicial
complexes (see \cite{Munkres} for a more detailed exposition). Let
us start with the definition of simplicial complex. The more usual
way to introduce this notion is to do it geometrically. First, let
us introduce the basic building blocks, that is, the notion of a
simplex. Given $n+1$ points $v_0,\dots,v_n$ in general position in
an Euclidean space, the $n$-simplex $\sigma$ generated by them,
$\sigma=(v_0,\dots,v_n)$, is defined as their convex hull. A
simplex $\sigma$ contains lower dimensional simplices $\tau$,
denoted by $\tau\leq \sigma$, called faces just by considering the
corresponding simplices generated by any subset of its vertices. A
simplicial complex is a collection of simplices satisfying two
conditions:

\begin{itemize}

\item Every face of simplex in a complex is also a simplex of the complex.

\item If two simplices in a complex intersect, then the intersection is also a simplex of the complex.

\end{itemize}

Alternatively, it is possible to define a simplicial complex
abstractly. It will avoid confusion when
parts of the simplicial structure are removed.

Given a finite set $[n]:=\{0,1,2,3,\ldots, n\}$, an \emph{abstract simplicial complex $K$ on $[n]$} is a collection of subsets of $[n]$ such that:

\begin{itemize}
\item If $\sigma \in K$ and $\tau \subseteq \sigma$, then $\tau \in K$.

\item $\{i\} \in K$ for every $i\in [n]$.

\end{itemize}

The set $[n]$ is called the \emph{vertex set} of $K$ and the elements $\{i\}$ are called \emph{vertices} or \emph{$0$-simplices.}   We also may write $V(K)$ to denote the vertex set of $K$.

An element $\sigma \in K$ of cardinality $i+1$ is called an \emph{$i$-dimensional simplex} or \emph{$i$-simplex} of $K$.   The \emph{dimension} of $K$, denoted $\dim(K)$, is the maximum of the dimensions of all its simplices.

\begin{example}  Let $n=5$ and $V(K):=\{0,1,2,3,4,5\}$.\par  \noindent Define $K:=\{\{0\},\{1\}, \{2\}, \{3\}, \{4\}, \{5\},  \{1,2\}, \{1,4\}, \{2,3\},
 \{1,3\}, \{3,4\}, \{3,5\},$\par \hspace*{1.55cm}$\{4,5\}, \{3,4,5\}, \emptyset\}$.\par
 This abstract simplicial complex may be regarded geometrically as follows:

\begin{center}
\includegraphics[width=.3\linewidth]{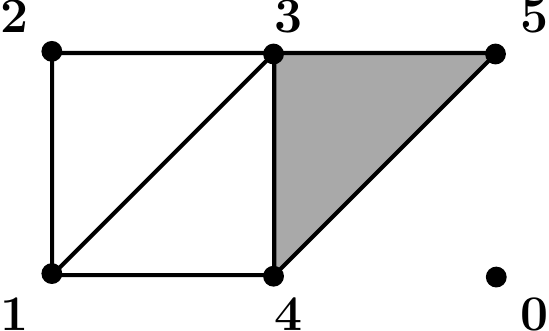}
\end{center}

\end{example}

Further definitions are in order.

\begin{definition}  A \emph{subcomplex} $L$ of $K$, denoted $L\subseteq K$, is a subset $L$ of $K$ such that $L$ is also a simplicial complex.
\end{definition}

We use $\sigma^{(i)}$ to denote a simplex of dimension $i$, and we write $\tau < \sigma^{(i)}$ to denote any subsimplex of $\sigma$ of dimension strictly less than $i$. If $\sigma, \tau\in K$ with $\tau < \sigma$, then $\tau$
is a \emph{face} of $\sigma$ and $\sigma$ is a \emph{coface} of $\tau$.  A simplex of $K$ that is not properly contained in any other simplex of $K$ is called a \emph{facet} of $K$.\\

At this point we recall a key concept in simple-homotopy theory: the notion of simplicial collapse \cite{C-73}.

\begin{definition}  Let $K$ be a simplicial complex and suppose that there is a pair of simplices $\sigma^{(p)}<\tau^{(p+1)}$ in $K$ such that $\sigma$ is a face of $\tau$ and $\sigma$ has no other cofaces.
Such a pair $\{\sigma, \tau\}$ is called a \emph{free pair.}. Then
$K-\{\sigma,\tau\}$ is a simplicial complex called an
\emph{elementary collapse} of $K$ (see
Figure~\ref{standard_collapse}).  The action of collapsing is
denoted $K \searrow K-\{\sigma, \tau\}$.

More generally, $K$ is said to \emph{collapse} onto $L$ if $L$ can
be obtained from $K$ through a finite series of elementary
collapses, denoted $K \searrow L$. In the case where $L=\{v\}$ is
a single vertex, we say that $K$ is  \emph{collapsible}.
\end{definition}

Now we introduce some basic subcomplexes related to a vertex. They play analogous role in the simplicial setting as the closed ball, sphere, and open ball play in the continuous approach.

\begin{definition} Let $K$ be a simplicial complex and let $v\in K$ be a vertex.

The \emph{star of $v$ in $K$}, denoted $\st(v)$, is the subcomplex of simplices $\sigma \in K$ such that $\sigma\cup \{v\}\in K$.

The \emph{link of $v$ in $K$}, denoted $\lk(v)$, is the subcomplex of $\st(v)$ of simplices which do not contain $v$.

The \emph{open star of $v$ in $K$} is $\st^o(v):=\st(v)-\lk(v)$. Note that $st^o(v)$ is not a simplicial subcomplex.

Finally, given a
simplicial complex $K$ and a vertex $v\notin K$, the \emph{cone} $vK$ is
the simplicial complex whose simplices are $\{v_0,\dots,v_n\}$ and $\{v,v_0,\dots,v_n\}$
where $\{v_0,\dots,v_n\}$ is any simplex of $K$.
\end{definition}

It is important to point out that simplicial collapses are not
simplicial maps in general. This suggests that it is natural to
consider a special kind of collapse which is a simplicial map. The
notion of strong collapse, introduced by Barmak and Minian in
\cite{BARMAK2011,B-M-12}, satisfies this requirement.

\begin{definition} Let $K$ be a simplicial complex and suppose there exists a pair of vertices $v,v^\prime\in K$ such that every
maximal simplex containing $v$ also contains $v^\prime$. Then we
say that $v^\prime$ \emph{dominates} $v$ and $v$ is
\emph{dominated by} $v^\prime$.

If $v$ is dominated by $v^\prime$ then the inclusion $i: K - \{v\}
\to K$ is a strong equivalence. Its homotopical inverse is the
retraction $r \colon K \to K - \{v\}$ which is the identity on $K
- \{v\}$ and such that $r(v) = v^\prime$. This retraction is
called an \emph{elementary strong collapse} from $K$ to $K -
\{v\}$, denoted by $K \searrow\searrow K - \{v\}$.

A \emph{strong collapse} is a finite sequence of elementary strong
collapses. The inverse of a strong collapse is called a
\emph{strong expansion} and two complexes $K$ and $L$ \emph{have
the same strong homotopy type} if there is a sequence of strong
collapses and strong expansions that transform $K$ into $L$.
\end{definition}

See Figure~\ref{strong_collapse} to illustrate the above notions.

Equivalently, $v$ is a dominated vertex if and only if its link is
a cone \cite{BARMAK2011}.

\section{Strong discrete Morse theory}\label{Strong discrete Morse theory}

We are now ready to introduce a key object of our study, that is,
discrete Morse functions in the R. Forman's sense \cite{F-98}.
More precisely, we are interested in a generalized notion
of critical object suitable for codifying the strong homotopy type
of a complex.

\begin{definition}  Let $K$ be a simplicial complex. A \emph{discrete Morse function} $f\colon K \to \RR$ is a function satisfying for any $p$-simplex
$\sigma\in K$:
\begin{enumerate}
\item[(M1)] $\sharp\left(\{\tau^{(p+1)}>\sigma : f(\tau)\leq f(\sigma)\}\right) \leq 1$.
\item[(M2)] $\sharp\left(\{\upsilon^{(p-1)}<\sigma : f(\upsilon)\geq f(\sigma)\}\right) \leq 1$.
\end{enumerate}
\end{definition}

Roughly speaking, we can say that it is a weakly increasing
function which satisfies the property that if $f(\sigma)=f(\tau)$,
then one of both simplices is a maximal coface of the other one.

\begin{definition}
A \emph{critical simplex} of $f$ is a simplex $\sigma$ satisfying:

\begin{enumerate}
\item[(C1)] $\sharp\left(\{\tau^{(p+1)}>\sigma : f(\tau)\leq f(\sigma)\}\right) =0$.
\item[(C2)] $\sharp\left(\{\upsilon^{(p-1)}<\sigma : f(\upsilon)\geq f(\sigma)\}\right)=0$.
\end{enumerate}
\end{definition}

If $\sigma$ is a critical simplex, the number $f(\sigma)\in \RR$ is called a \emph{critical value}. Any simplex that is not critical is called a \emph{regular simplex} while any output value of the discrete Morse function
which is not a critical value is a \emph{regular value}. The set of critical simplices of $f$ is denoted by $\Crit(f)$.\\

Given any real number $c$, the \emph{level subcomplex} of $f$ at
level $c$, $K(c)$, is the subcomplex of $K$ consisting of all
simplices $\tau$ with $f(\tau)\leq c$, as well as all of their
faces, that is,

$$ K(c)=\bigcup_{f(\tau)\leq c}\bigcup_ {\sigma \leq \tau} \sigma .$$

Any discrete Morse function induces a gradient vector field.

\begin{definition}\label{GVF}  Let $f$ be a discrete Morse function on $K$.  The \emph{induced gradient vector field} $V_f$ or $V$ when the context is clear, is defined by the following set of pairs of simplices:
$$V_f:=\{(\sigma^{(p)}, \tau^{(p+1)}) : \sigma<\tau, f(\sigma)\geq f(\tau)\}.
$$
\end{definition}

Note that critical simplices are easily
identified in terms of the gradient field as precisely
those simplices not contained in any pair in the gradient field.

\begin{definition}\label{min outside} Let $f\colon K\to \RR$ be a discrete Morse function and $V_f$ its induced gradient vector field.  For each vertex/edge pair
$(v,uv)\in V_f$, write $\mathrm{St}(v,u):=\st^o(v)\cap \st(u)$.
Define $m_v:=\min\{f(\tau): f(\tau)>f(uv), \tau \in
(K-\mathrm{St}(v,u))\bigcup (\Crit(f)\cap \mathrm{St}(v,u) )\}$.
\end{definition}

\begin{example}\label{cutoff illustration_a}
We illustrate Definition \ref{min outside}.  Let $K$ be the simplicial complex with
discrete Morse function $f$ given in Figure~\ref{DMT_example1a}.

\begin{figure}[htbp]
\begin{center}
\includegraphics[width=.95\linewidth]{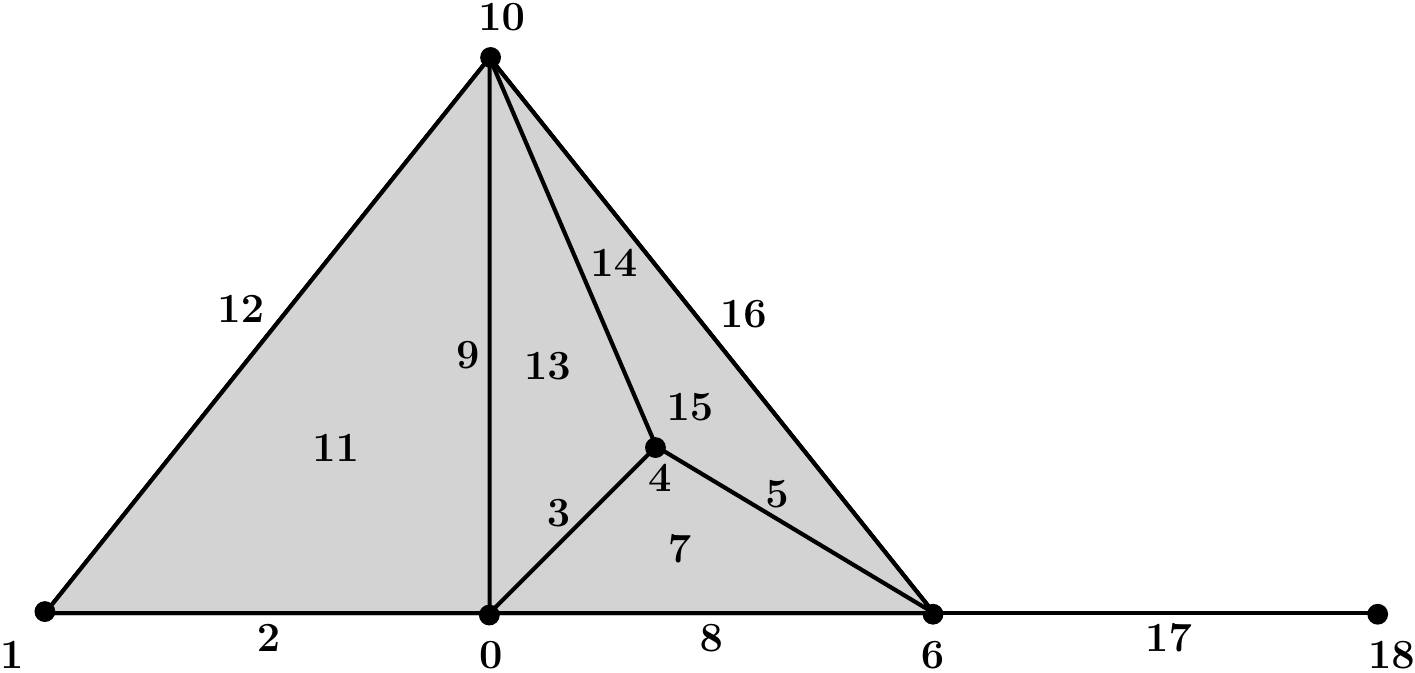}
\caption{A 2-dimensional simplicial complex with a discrete Morse
function.} \label{DMT_example1a}
\end{center}
\end{figure}

Observe that taking $v=f^{-1}(10)$ and $u=f^{-1}(0)$, we have
$(v,uv)\in V_f$ and
$$\mathrm{St}(v,u)=f^{-1}(\{9,10,11,12,13,14\})\,.$$ Since $\mathrm{St}(v,u)$ in Example~\ref{cutoff illustration_a} does not contain any
critical simplex, then determining $m_v$ consists on finding the
smallest value greater than $f(uv)$ outside of $\mathrm{St}(v,u)$.
In this case, $m_v=15$.
\end{example}

\begin{example}\label{cutoff illustration_b}
Now $f$ will be slightly modified to create a new discrete Morse
function $g$ on $K$ (Figure~\ref{DMT_example1b}). It is clear that $(v,uv)\in V_g$, but in this case $m_v=11$.

\begin{figure}[htbp]
\begin{center}
\includegraphics[width=.95\linewidth]{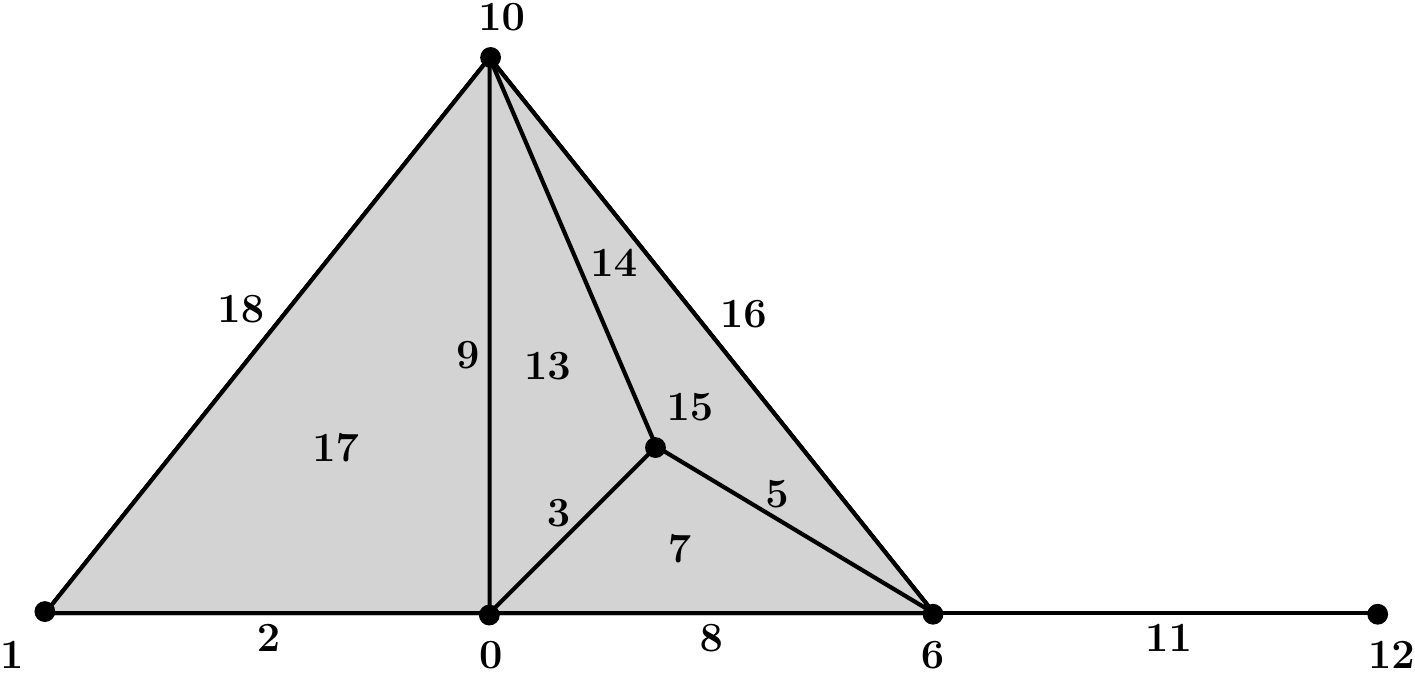}
\caption{The same 2-dimensional simplicial complex of
Figure~\ref{DMT_example1a} with another discrete Morse function.}
\label{DMT_example1b}
\end{center}
\end{figure}

\end{example}

Notice that the values that take both functions $f$ and $g$ on
some simplices coincide. It is interesting to point out that both
examples have the same induced gradient vector field and
consequently, contain the same critical simplices in the usual
Forman sense. This justifies the need to take into account
additional information in this new approach and thus a more general concept of critical object.

\begin{definition}\label{strong definition} Continuing with the notation used in Definition \ref{min outside}, define $l_v$ as the largest regular value in $f(\mathrm{St}(v,u))$ such that
\begin{itemize}
  \item $f(uv)\leq l_v\leq m_v$
  \item every maximal regular simplex of $K(l_v)\cap \mathrm{St}(v,u)$ contains the vertex $u$.
\end{itemize}

\end{definition}

\begin{definition}
Under the notation of Definition \ref{min outside}, the
\emph{strong collapse of $v$ under $f$},\linebreak denoted
$S^f_v$, is given by $S^{f}_v:=\{ (\sigma, \tau)\in V_f :
f(uv)\leq f(\tau)\leq l_v\}$ and the interval
$I(S^f_v)=[f(uv),l_v]$ is called the \emph{strong interval} of
$S_v$.

The elements of the set $C(V_f):=V_f-\bigcup S^f_v$ are the
\emph{critical pairs} of $f$ while each element in $\bigcup
S_v$ is a \emph{regular pair} of $f$.  If $(\sigma, \tau)$ is a
critical pair, the value $f(\tau)$ is a \emph{critical value} of
$f$.

A \emph{critical object} is either a critical simplex (in the
standard Forman sense) or a critical pair. The set of all critical
objects of $f$ is denoted by $\scrit(f)$. In order to avoid
confusion, we call the images of all critical objects (either in
the Forman sense or from a critical pair) \emph{strong critical
values}.
\end{definition}

\begin{remark}

It is worthwhile to mention that critical pairs are
detecting when a standard collapse has been made. Notice that it may happen either due to combinatorial
reasons (e.g. the non-existence of dominated vertices in the
corresponding subcomplex) or to a bad choice of the values of the
discrete Morse function (i.e. noise). This second option induces a bad ordering in
the way the standard collapses are made inside a potential strong one. The
key idea is that every elementary strong collapse should ideally
be made as a uninterrupted sequence of standard collapses. So
every time it is not made in this way, a new critical object
appears.
\end{remark}

\begin{example}
Define $f$ and $g$ as in Examples \ref{cutoff illustration_a} and
\ref{cutoff illustration_b}. We see that $c^f_v=14$ while
$l_v^g=10$.  Hence (by abuse of notation), $(16,15)$ is a critical
pair under $f$ while $(14,13)$, $(16,15)$ and $(18,17)$ are
critical pairs under $g$. Notice that the strong intervals are
$I(S^f_v)=[9,14]$ and $I(S^g_v)=[9,10]$, respectively.
\end{example}

\begin{example}\label{arg exam}
To further illustrate Definition \ref{strong definition}, we shall
consider two different discrete Morse functions defined on a
collapsible but non-strongly collapsible triangulation of the
$2$-disc. Notice that both of them have only a single critical
value (in the Forman sense), but very different numbers of
critical objects, due to a bad election in the ordering of normal
collapses of $g$.

Let $f$ be given in Figure~\ref{DMT_example2a}.

\begin{figure}[htbp]
\begin{center}
\includegraphics[width=.7\linewidth]{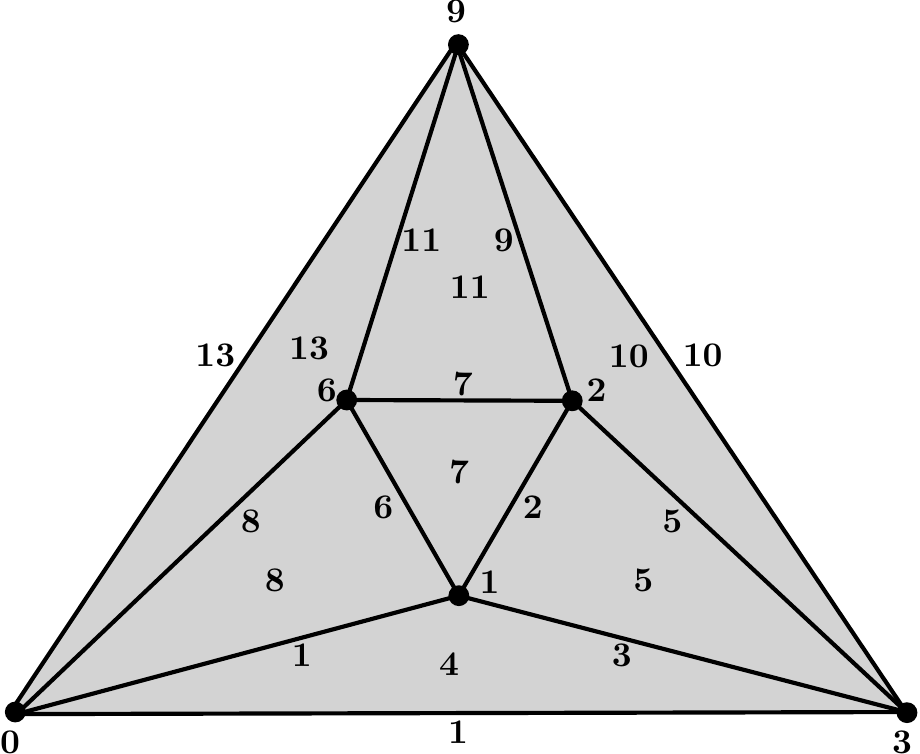}
\caption{A discrete Morse function defined on a non-strongly
collapsible triangulation of the $2$-disc.} \label{DMT_example2a}
\end{center}
\end{figure}

By abuse of notation, we refer to the simplices by their labeling
under the discrete Morse function (the fact that pairs in $V_f$
are given the same label should not cause confusion).  For each of
the pairs $(9,9), (6,6), (3,3),(2,2),(1,1)\in V_f$, we have the
corresponding values $l_9=11, l_6=8, l_3=5, l_2=2$, and $l_1=1$.
The corresponding strong collapses under the indicated vertices
are given by $$
\begin{array}{ccl}
  S^f_9 & = & \{ (9,9), (10,10), (11,11)\},\\
  S^f_6 & = & \{(6,6),(7,7), (8,8)\}, \\
  S^f_3 & = & \{(3,3), (4,4), (5,5)\}, \\
  S^f_2 & = & \{(2,2)\} \mbox{ and }\\
  S^f_1 & = & \{(1,1)\}\,.
\end{array}$$
Thus we obtain the following strong intervals: $$\begin{array}{c}
                                                   I(S^f_9)=[9,11], I(S^f_6)=[6,8], I(S^f_3)=[3,5], I(S^f_2)=[2,2]=\{ 2 \} \\
                                                   \mbox{ and } I(S^f_1)=[1,1]=\{ 1 \}\,.
                                                 \end{array}$$
Hence there is a single critical pair, namely, $(13,13)$, so that $\scrit(f)=\{0, (13,13)\}$.

It is interesting to point out that this discrete Morse function
can be considered as optimal in the sense that it minimizes the
number of critical objects, as we will see in Theorem \ref{sLS
theorem}.

Now let $g$ be the discrete Morse function given in
Figure~\ref{DMT_example2b}.

\begin{figure}[htbp]
\begin{center}
\includegraphics[width=.7\linewidth]{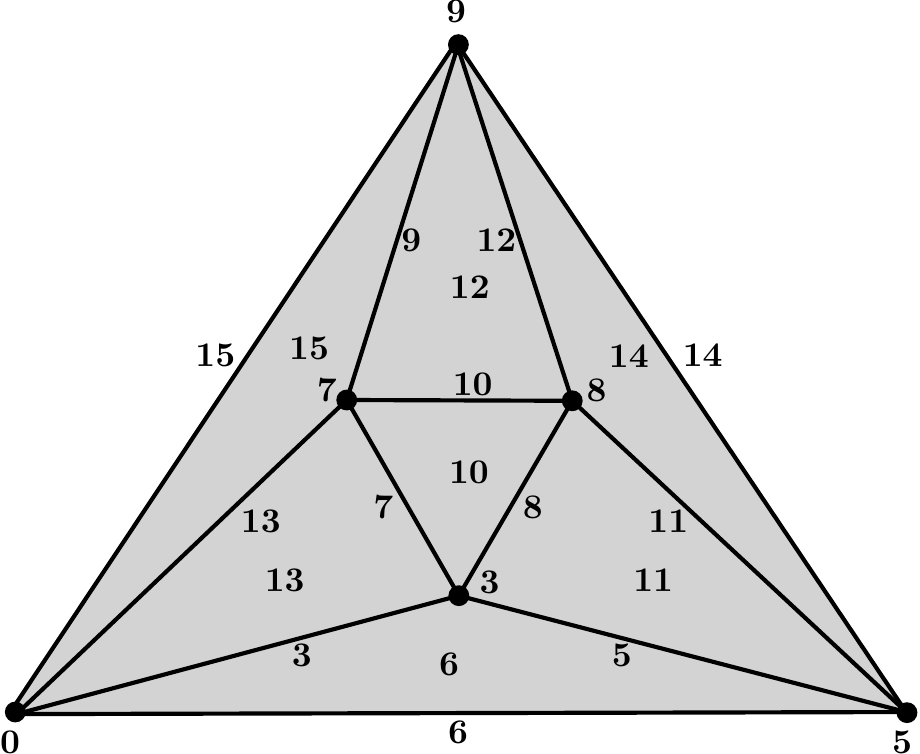}
\caption{A different discrete Morse function defined on a
non-strongly collapsible triangulation of the $2$-disc.}
\label{DMT_example2b}
\end{center}
\end{figure}

Now we consider the pairs $(9,9),(8,8),(7,7), (5,5), (3,3)\in V_g$
and obtain corresponding values $ l_9=10, l_8=8, l_7=7,l_5=6$ and
$l_3=3.$ The corresponding strong collapses are then given by
$$\begin{array}{ccl}
    S^g_9 & = & \{(9,9)\}, \\
    S^g_8 & = & \{(8,8)\}, \\
    S^g_7 & = & \{(7,7)\}, \\
    S^g_5 & = & \{(5,5),(6,6)\}\mbox{ and } \\
    S^g_3 & = & \{(3,3)\}\,.
  \end{array}$$
It follows that the strong intervals are: $$\begin{array}{c}
    I(S^g_9)=[9,9]=\{ 9 \}, I(S^g_8)=[8,8]=\{ 8 \}, I(S^g_7)=[7,7]=\{ 7 \}, I(S^g_5)=[5,6]\\
    \mbox{ and } I(S^g_3)=[3,3]=\{3 \}\,.
  \end{array}$$
We conclude that $\scrit(g)=\{ 0, (15,15), (13,13), (14,14), (10,10),(12,12), (11,11)\}$ for a total of $7$ critical
objects.

\end{example}

The following Theorem is the analogue of Forman's classical result
Theorem 3.3 of \cite{F-98} and also analogous to the classical
result in the smooth setting establishing homotopy equivalence of
sublevel sets by assuming non-existence of critical values.

\begin{theorem}\label{analagous collapse} Let $f$ be a discrete Morse function on $K$ and suppose that $f$ has no strong critical values on $[a,b]$. Then $K(b)\searrow\searrow K(a)$. 
In particular, if $I(S^f_v)=[f(uv),l_v]$ is a strong interval for
some vertex $v\in K$, then $K(l_v)\searrow\searrow K(f(uv))$.
\end{theorem}

\begin{proof} Let us consider the interval $[a,b]$ such that it does not contain any strong critical value.
If additionally this interval does not contain any strong interval, then we conclude that $K(l_1)=K(l_2)$.
Hence, assume that $[a, b]$ contains strong intervals. We may
then partition $[a,b]$ into subintervals such that each one
of them contains exactly one strong interval. Let us suppose
$I(S_v)=[f(uv), l_v]\subseteq [a, b]$ is the unique strong
interval in $[a, b]$. Again, since $f$ does not take values
in $(a, f(uv))$  or $(l_v,b)$, it follows that
$K(a)=K(f(uv))$ and $K(l_v)=K(b)$. Thus it suffices to show
that $K(l_v)\searrow\searrow K(f(uv))$. By definition of $l_v$,
$u$ is contained in every maximal regular simplex of $K(l_v) \cap
\mathrm{St}(v, u).$ Furthermore, since $l_v< m_v$, all the new
simplices that are attached from $K(f(uv))$ to $K(l_v)$ are
contained within $\mathrm{St}(v, u)$. Hence
$K(l_v)-K(f(uv)-\epsilon)$ is an open cone with apex $u$ and thus,
$K(l_v)\searrow\searrow K(f(uv))$.

\end{proof}

\section{Simplicial Lusternik--Schnirelmann category}\label{simplicial Lusternik--Schnirelmann category}

In this section, we recall the fundamental definitions and results found in \cite{F-M-V}.

\begin{definition} Let $K,L$ be simplicial complexes.  We say that two simplicial maps\linebreak $\phi, \psi \colon K\to L$ are \emph{contiguous},
denoted $\phi\sim_c \psi$, if for any simplex $\sigma \in K$, we
have that $\phi(\sigma)\cup \psi(\sigma)$ is a simplex of $L$.
Because this relation is reflexive and symmetric but not
transitive, we say that $\phi$ and $\psi$ are in the same
\emph{contiguity class}, denoted $\phi\sim \psi$, if there is a
sequence $\phi=\phi_0\sim_c \phi_1 \sim_c\ldots \sim_c \phi_n
=\psi$ of contiguous simplicial maps $\phi_i\colon K\to L, 0\leq i
\leq n$. A map $\phi \colon K \to L$ is a \emph{strong
equivalence} if there exists $\psi\colon  L \to K$ such that $\psi
\phi \sim \id_K$ and $\phi \psi \sim \id_L$. In this case, we say
that $K$ and $L$ are \emph{strongly equivalent}, denoted by $K
\sim L$.
\end{definition}

There is a nice link between strong equivalences and strong
collapses.

\begin{theorem}\cite[Cor. 2.12]{B-M-12} Two complexes $K$ and $L$ have the same strong homotopy type if and only if $K \sim L$.
\end{theorem}

\begin{definition}\cite{F-M-V}\label{scat definition} Let K be a simplicial complex. We say that the subcomplex $U \subseteq K$ is
\emph{categorical} in $K$ if there exists a vertex $v \in K$
such that the inclusion $i\colon  U \to K$ and the constant map
$c_v \colon U \to K$ are in the same contiguity class.  The
\emph{simplicial L--S category}, denoted \emph{$\scat (K)$}, of
the simplicial complex $K$, is the least integer $m \geq 0$ such
that $K$ can be covered by $m + 1$ categorical subcomplexes.
\end{definition}

One of the basic results of classic Lusternik-Schnirelmann theory
states that the L--S category is homotopy invariant. Next result shows
that simplicial L--S category satisfies the analogous property in
the discrete setting.

\begin{theorem}\label{FMV}\cite[Theorem 3.4]{F-M-V} Let $K \sim L$ be two strongly equivalent complexes. Then $\scat(K) = \scat(L)$.
\end{theorem}

We refer the interested reader to the papers \cite{F-M-V,F-M-M-V} for a
detailed study of this topic.

\subsection{Simplicial Lusternik--Schnirelmann theorem}

Our main result is the following simplicial version of the
Lusternik-Schnirelmann Theorem:

\begin{theorem}\label{sLS theorem} Let $f\colon K\to \RR$ be a discrete Morse function.  Then
$$\scat(K)+1\leq \sharp(\scrit(f))\,.$$
\end{theorem}

\begin{proof}
For any natural number $n$, define $c_n :=\min\{a\in \RR :
\scat(K(a))\geq n-1\}$. We claim that $c_n$ is a strong critical
value of $f$.  If $c_n$ is a regular value, then it is either
contained in a strong interval or it is not.  If $c_n$ is
contained in a strong interval $I(c_v)$, then by Theorem
\ref{analagous collapse}, we have $\scat(K(c_v))=\scat(K(c_n))$,
contradicting the minimality of $c_n$.  Otherwise, $c_n$ is
outside a strong interval. Then, by Theorem \ref{analagous
collapse} there exists $\epsilon >0$ such that
$K(c_n)\searrow\searrow K(c_n-\epsilon)$. By Theorem \ref{FMV},
$\scat(K(c_n))=\scat(K(c_n-\epsilon))$. But $c_n>c_n-\epsilon$ and
$c_n$ was the minimum value such that $\scat(K(c_n))=n-1$, which
is a contradiction.  Thus each $c_n$ is a strong critical value of
$f$.

\smallskip

We now prove by induction on $n$ that $K(c_n)$ must contain at
least $n$ critical objects.  By the well-ordering principle, the
set $\im(f)$ has a minimum, say $f(v)=0$ for some $0$-simplex
$v\in K$.  For $n=1$, $c_1=0$ so that $K(c_1)$ contains $1$
critical object.  For the inductive hypothesis, suppose that
$K(c_n)$ contains at least $n$ strong critical objects.  Since all
the strong critical values of $f$ are distinct, $c_n<c_{n+1}$ so
that there is at least one new critical object in
$f^{-1}(c_{n+1})$. Thus $K(c_{n+1})$ contains at least $n+1$
critical objects.  Hence if $c_1<c_2<\ldots < c_{\scat(K)+1}$ are
the critical objects, then $K(c_{\scat(K)+1})\subseteq K$ contains
at least $\scat(K)+1$ critical objects.  Thus $\scat(K)+1\leq
\sharp(\scrit(f))$.
\end{proof}

\begin{example}\label{LS equality}
We give an example where the upper bound of Theorem \ref{sLS
theorem} is attained. Let $A$ be the several times considered
non-strongly collapsible triangulation of the $2$-disc with the
discrete Morse function $f$ given in Example \ref{arg exam}.  This
satisfies $\sharp(\scrit(f))=2$. Since $A$ has no dominating
vertex, $\scat(A)>0$, whence $\scat(A)+1=\sharp(\scrit(f))=2$.\\

Notice that just adding one simplex to $A$, the above equality
turns into a strict inequality and thus, the number of critical
objects may increase while the simplicial category keeps the same.
To see this, let us consider $A'$ as the clique complex of $A$,
that is, we add to $A$ the triangle generated by its three
bounding vertices. This is a simplicial $2$-sphere, so by means of
Morse inequalities, it follows that every discrete Morse function
$f$ defined on $A'$ has at least two critical simplices: one
critical vertex (global minimum) and a critical triangle (global
maximum). In addition, since $A'$ is a simplicial $2$-sphere, then
it does not contain any dominated vertex. Moreover, after removing
the critical triangle (and hence obtaining $A$) no new dominated
vertices appear, so at least one critical pair arises in order to
collapse $A$ to a subcomplex containing dominated vertices. Hence,
we conclude that that any discrete Morse function $f\colon A' \to
\RR$ must have at least $3$ critical objects, which could be
considered optimal in the sense that it minimizes the number of
critical objects. Furthermore, it is easy to cover $A'$ with $2$
categorical sets so that $\scat(A')=1$. Hence $\scat(A')+1=2<3\leq
\sharp(\scrit(f))$.
\end{example}

\bibliographystyle{amcjoucc}

\begin{thebibliography}{99}

\bibitem{AS}
S. Aaronson, N. Scoville,
Lusternik-Schnirelmann for simplicial complexes,
\emph{Illinois J. Math.} \textbf{57} (3) (2013), 743--753.

\bibitem{BARMAK2011}
J.A. Barmak,
\emph{Algebraic topology of finite topological spaces and applications},
Lecture Notes in Mathematics, 2032, Springer, Heidelberg, 2011.

\bibitem{B-M-12}
J.A. Barmak and E.G. Minian,
Strong homotopy types, nerves and collapses,
\emph{Discrete Comput. Geom.} \textbf{47} (2) (2012), 301--328.

\bibitem{C-73}
M. M. Cohen,
\emph{A course in simple-homotopy theory},
Graduate Texts in Mathematics, Vol. 10. Springer-Verlag, New York-Berlin, 1973.

\bibitem{CLOT}
O. Cornea, G. Lupton, J. Oprea, D. Tanr\'e,
\emph{Lusternik-Schnirelmann Category}, Mathematical Surveys and
Monographs, 103. American Mathematical Society, Providence, RI,
2003.


\bibitem{C-G-N2016}
J. Curry, R. Ghrist, V. Nanda,
Discrete Morse Theory for Computing Cellular Sheaf Cohomology,
\emph{Found. Comput. Math} \textbf{16} (2016), 875--897.

\bibitem{F-M-M-V}
D. Fern\'andez-Ternero, E. Mac\'{\i}as-Virg\'os, E. Minuz and J.A. Vilches,
Simplicial Lusternik-Schnirelmann category,
arXiv:1605.01322[math.AT], 2016.

\bibitem{F-M-V}
D. Fern\'andez-Ternero, E. Mac\'{\i}as-Virg\'os, J.A. Vilches,
Lusternik-Schnirelmann category of simplicial complexes and finite spaces,
\emph{Topology Appl.} \textbf{194} (2015), 37--50.

\bibitem{F-98}
R. Forman,
Morse Theory for cell complexes,
\emph{Adv. Math.} \textbf{134} (1) (1998), 90--145.

\bibitem{L-S}
L. Lusternik and L. Schnirelmann,
\emph{M\'ethodes Topologiques dans les Probl\`emes Variationnels},
Herman, Paris, 1934.

\bibitem{M-63}
J. Milnor,
\emph{Morse theory},
Based on lecture notes by M. Spivak and R. Wells,
Annals of Mathematics Studies, No. 51 Princeton University Press, Princeton, N.J. 1963.

\bibitem{Munkres} J. R. Munkres, \emph{Elements of Algebraic
Topology}, Addison Wesley, Menlo Park, CA. 1984.


\bibitem{Palais-66}
R. S. Palis, Lusternik--Schnirelmann theory on Banach manifolds,
\emph{Topology} \textbf{5} (1966), 115--132.


\bibitem{Real-15}
R. Reina-Molina, D. D\'{\i}az-Pernil, P. Real and A. Berciano,
Membrane parallelism for discrete Morse theory applied to digital
images, \emph{Appl. Algebra Engrg. Comm. Comput.} \textbf{26}
(1-2) (2015), 49--71.

\end{thebibliography}

\bigskip

\address{
\noindent {\sc D.~Fern\'andez-Ternero}.
\\Dpto. de Geometr\'{\i}a y Topolog\'{\i}a, Universidad de Sevilla, Spain.\\}
\email{desamfer@us.es}

\medskip

\address{
\noindent {\sc E.~Mac\'ias-Virg\'os}.
\\{Dpto. de Matem\'aticas,} Universidade de San\-tia\-go de Compostela, Spain.\\}
\email{quique.macias@usc.es}

\medskip

\address{
\noindent {\sc N.~A.~Scoville}.
\\Mathematics and Computer Science Ursinus College, 610 E Main Street, Collegeville PA
19426, USA\\}
\email{nscoville@ursinus.edu}

\medskip

\address{
\noindent {\sc J.A.~Vilches}.
\\Dpto. de Geometr\'{\i}a y Topolog\'{\i}a, Universidad de Sevilla, Spain.\\}
\email{vilches@us.es}

\end{document}